\documentclass{amsart}%
\usepackage{amsfonts}
\usepackage{amsmath}
\usepackage{amssymb}
\usepackage{graphicx}%
\providecommand{\U}[1]{\protect\rule{.1in}{.1in}}
\newtheorem{theorem}{Theorem}
\theoremstyle{plain}

\newtheorem{corollary}{Corollary}

\newtheorem{definition}{Definition}
\newtheorem{example}{Example}

\newtheorem{lemma}{Lemma}

\newtheorem{problem}{Problem}

\newtheorem{remark}{Remark}

\numberwithin{equation}{section}
\begin{document}
\title[ ]{A note on the isotonic vector-valued convex functions}
\author{Constantin P. Niculescu}
\address{Department of Mathematics, University of Craiova\\
Craiova 200585, Romania}
\email{constantin.p.niculescu@gmail.com}
\author{Octav Olteanu}
\address{Department of Mathematics-Informatics\\
University Politehnica of Bucharest, \linebreak Bucharest 060042, Romania}
\email{octav.olteanu50@gmail.com}
\thanks{Corresponding author: Constantin P. Niculescu.}
\date{April 14, 2020}

\begin{abstract}
The property of isotonicity of a continuous convex function defined on the
entire space or only on the positive cone is characterized via
subdifferentials. Numerous examples illustrating the obtained results are included.

\end{abstract}
\maketitle

\section{Introduction}

An important notion in nonlinear functional analysis is that of convex
function. The existing literature is really vast, a simple search on Google
Scholars offering thousands of titles, starting with the pioneering work on
sublinear operators done by Kantorovich \cite{Kan}, Rubinov \cite{Rub1970},
\cite{Rub1977}, Ioffe and Levin \cite{IL1972}, Levin \cite{Lev1972}, Linke
\cite{Lin1972}, Feldman \cite{Fe1975} and Kutateladze \cite{K1979}, to which
we should add the monographs of Rockafellar \cite{Roc}, J. M. Borwein and J.
Vanderwerff \cite{BV2009}, H\"{o}rmander \cite{Hor94}, Kusraev \cite{Kus2000},
Kusraev and Kutateladze \cite{KK1995}, Simon \cite{Simon2011}, Z\u{a}linescu
\cite{Zal2002} etc.

The definition of a convex function needs the existence of an order structure
on its codomain, usually supposed to be a real ordered vector space, or a real
vector space with a richer structure (like an ordered Banach space, a Banach
lattice, an order complete Banach lattice etc.). We always assume that the
positive cone $X_{+}$ of the ordered linear space $X$ under attention is
proper $(-X_{+}\cap X_{+}=\left\{  0\right\}  )$ and generating $(X=X_{+}%
-X_{+});$ in the case of ordered Banach spaces, these assumptions are
supplemented by the requirement that $X_{+}$ is closed and $0\leq x\leq y$ in
$X$ implies $\left\Vert x\right\Vert \leq\left\Vert y\right\Vert .$ A
convenient way to emphasize the properties of ordered Banach spaces is that
described by Davies in \cite{Davies1968}. According to Davies, a real Banach
space $E$ endowed with a closed and generating cone $E_{+}$ such that%
\[
\left\Vert x\right\Vert =\inf\left\{  \left\Vert y\right\Vert :y\in E,\text{
}-y\leq x\leq y\right\}  \text{\quad for all }x\in E,
\]
is called a \emph{regularly ordered Banach space}%
\index{space!regularly ordered}%
.\emph{ }Examples are the Banach lattices and some other spaces such as
$\operatorname*{Sym}(n,\mathbb{R)}$, the ordered Banach space of all $n\times
n$-dimensional symmetric matrices with real coefficients. The norm of a
symmetric matrix $A$ is defined by the formula%
\[
\left\Vert A\right\Vert =\sup_{\left\Vert x\right\Vert \leq1}\left\vert
\langle Ax,x\rangle\right\vert ,
\]
and the positive cone $\operatorname*{Sym}^{+}(n,\mathbb{R})$ of
$\operatorname*{Sym}(n,\mathbb{R})$ consists of all symmetric matrices $A$
such that $\langle Ax,x\rangle\geq0$ for all $x.$

Every ordered Banach space can be renormed by an equivalent norm to become a
regularly ordered Banach space. For details, see Namioka \cite{Nam}. Some
other useful properties are listed below.

\begin{lemma}
\label{lem1}Suppose that $E$ is a regularly ordered Banach space. Then:

$(a)$ There exists a constant $C>0$ such that every element $x\in E$ admits a
decomposition of the form $x=u-v$ where $u,v\in E_{+}$ and $\left\Vert
u\right\Vert ,\left\Vert v\right\Vert \leq C\left\Vert x\right\Vert .$

$(b)$ The dual space of $E,$ $E^{\ast},$ when endowed with the dual cone
\[
E_{+}^{\ast}=\left\{  x^{\ast}\in E^{\ast}:x^{\ast}(x)\geq0\text{ for all
}x\in E_{+}\right\}
\]
is a regularly ordered Banach space.

$(c)\ x\leq y$ in $E$ is equivalent to $x^{\ast}(x)\leq x^{\ast}(y)$ for all
$x^{\ast}\in E_{+}^{\ast}.$
\end{lemma}

\begin{proof}
The assertion $(a)$ is trivial. For $(b)$, see Davies \cite{Davies1968}, Lemma
2.4. The assertion $(c)$ is an easy consequence of the Hahn-Banach separation
theorem; see \cite{NP2018}, Theorem 2.5.3, p. 100.
\end{proof}

The positive linear operators acting on ordered Banach spaces are necessarily
continuous. The following result appears in a slightly different form in the
book of Schaefer and Wolff \cite{SW1999}, Theorem 5.5 $(ii)$, p. 228. Since
the way to cover the details is not obvious, we have included here our proof.

\begin{lemma}
\label{lem1*}Every positive linear operator $T:E\rightarrow F$ acting on
ordered Banach spaces is continuous.
\end{lemma}

\begin{proof}
According to a remark above we may assume that both spaces $E$ and $F$ are
regularly ordered Banach spaces. A moment's reflection shows that the property
of continuity of $T$ is equivalent to the fact that $T$ maps the bounded
sequences into bounded sequences. Taking into account Lemma \ref{lem1} $(a)$,
one can restate this as $T$ maps positive bounded sequences into positive
bounded sequences. Thus, if $T$ would be not continuous, then there exists a
sequence $(x_{n})_{n}$ of positive norm-1 elements of $E$ such that
$\left\Vert T(x_{n})\right\Vert \rightarrow\infty.$ According to the
Banach-Steinhaus theorem (the uniform boundedness principle) this implies the
existence of a continuous linear functional $y^{\ast}\in F^{\ast}$ such that
$\left\vert y^{\ast}\left(  T(x_{n})\right)  \right\vert \rightarrow\infty.$
Taking into account the assertions $(b)$ and $(a)$ of Lemma \ref{lem1}, we may
also assume that $y^{\ast}\geq0.$ Passing to a subsequence if necessary we
arrive at the situation when $y^{\ast}\left(  T(x_{n})\right)  \geq2^{n}$ for
all $n\in\mathbb{N}.$ Put $x=\sum_{n=1}^{\infty}2^{-n}x_{n}.$ Then
\[
y^{\ast}(T(x))\geq\sum_{k=1}^{n}2^{-k}y^{\ast}(T(x_{k}))\geq n
\]
for all $n\in\mathbb{N},$ a contradiction.
\end{proof}

A thorough presentation of the theory of Banach lattices is available in the
book of Meyer-Nieberg \cite{MN}, while additional information on the case of
ordered Banach spaces (and other ordered topological vector spaces) can be
found in the book of Schaefer and Wolff \cite{SW1999}.

In the algebraic context, a map $\Phi$ from a convex subset $C$ of the vector
space $E$ and taking values in an ordered vector space $F$ is called a
\emph{convex function} if%
\[
\Phi((1-\lambda)x+\lambda y)\leq(1-\lambda)\Phi(x)+\lambda\Phi(y)\text{\quad
for all }x,y\in C\text{ and }\lambda\in\lbrack0,1].
\]

In the topological context (when $E$ is a Banach space and $F$ is an ordered
Banach space$)$ we will be interested in the \emph{continuous convex
functions.} Notice that unlike the finite dimension case, not every convex
function acting on an infinite dimensional space is continuous. Indeed, every
infinite dimensional normed space admits a discontinuous linear functional;
see \cite{NP2018}, Remark 3.1.13, p. 113.

The aim of the present paper is to characterize the convex functions $\Phi$
which are \emph{isotone} (or order preserving) in the sense that
\[
x\leq y\text{ implies }\Phi(x)\leq\Phi(y).
\]
The use of this terminology is motivated by the need to avoid any possible
confusion with another concept of \emph{monotone} map, frequently used in
variational analysis; see \cite{BV2009} and \cite{Roc} (as well as the
references therein).

As was noticed by various authors (see, for example, Amann \cite{Amann1974},
\cite{Amann1976}), important results in nonlinear analysis are related to the
class of isotone convex functions. Also this class of functions offers the
natural framework for many interesting inequalities. See \cite{APPV} and the
references therein.

In this paper the isotonicity of a continuous convex function is characterized
in terms of positivity of its subgradients. For the convenience of the reader,
a proof of the property of subdifferentiability of continuous convex functions
based on the extension properties of linear operators dominated by a convex
function is presented in Section 2. The main new results are presented in
Section 3. Theorems 4 and 5 give characterizations of the isotonicity of
convex functions defined respectively on the entire space and on the positive
cone, provided its interior is nonempty. The case where the interior of the
positive cone is empty is covered by Theorem 6.

Examples illustrating the richness of the class of isotone convex functions
make the objective of Section 4 and Section 5, the last being designated to a
description of the isotone sublinear operators associated to Choquet's integral.

\section{The Generalized\emph{ }Hahn-Banach Extension Theorem and Its
Consequences}

In what follows $C$ is a convex subset of a Banach space $E$ and $F$ is an
ordered Banach space (or a similar space with stronger properties).

Given a map $\Phi:C\rightarrow F,$ we call \emph{subgradient} \emph{of} $\Phi$
\emph{at a point} $x_{0}\in C$ every continuous linear operator
$T:E\rightarrow F$ such that%
\begin{equation}
\Phi(x)\geq\Phi(x_{0})+T(x-x_{0})\text{\quad for all }x\in C. \label{eqsdiff}%
\end{equation}
The set $\partial_{x_{0}}\Phi,$ of all such subgradients is referred to as the
\emph{subdifferential of} $\Phi$ \emph{at} $x_{0}.$ Geometrically, the
subgradient inequality (\ref{eqsdiff}) asserts that $\Phi$ admits a continuous
affine minorant $A(x)=\Phi(x_{0})+T(x-x_{0})$ such that $A(x_{0})=\Phi
(x_{0}).$

The set valued map $x_{0}\rightarrow\partial_{x_{0}}\Phi$ is referred to as
the \emph{subdifferential} of $\Phi.$ The following basic result is due to
Valadier \cite{Val}, p. 68, and Zowe \cite{Zowe}, p. 286.

\begin{theorem}
\label{thm1}Suppose that $C$ is a convex subset of a Banach space $E$ and $F$
is an order complete Banach space. Then the subdifferential of any continuous
convex function $\Phi:C\rightarrow F$ at an interior point $x_{0}\in C$ is a
nonempty convex set.
\end{theorem}

\begin{corollary}
\label{cor1}Suppose that $C$ is an open convex subset of the Banach space $E$
and $F$ is an order complete Banach space. Then every continuous convex
function $\Phi:C\rightarrow F$ is the upper envelope of its continuous affine minorants.
\end{corollary}

Valadier's argument for the non emptiness of the subdifferential is based on
the existence of directional derivatives%
\[
\Phi^{\prime}(x_{0};h)=\lim_{t\rightarrow0+}\frac{\Phi(x_{0}+th)-\Phi(x_{0}%
)}{t},
\]
and the fact that a linear operator $T:E\rightarrow F$ is a subgradient of
$\Phi$ at $x_{0}$ if and only if $T(h)\leq\Phi^{\prime}(x_{0};h)$ for every
$h\in E.$ He assumed $F$ to be an order complete Banach lattice, a fact
relaxed by Zowe who used a generalization of the analytical form of the
Hahn-Banach Theorem stated below as Theorem \ref{thm2}.

The converse of Theorem \ref{thm1} also works: the only continuous functions
$\Phi:C\rightarrow F$ with the property that $\partial_{x_{0}}\Phi
\neq\emptyset$ at any interior point $x_{0}$ of $C$ are the convex ones. See
\cite{NP2018}, Theorem 3.3.6, p. 124.

\begin{theorem}
[The Generalized\emph{ }Hahn-Banach Extension Theorem]\label{thm2}Let $\Phi$
be a convex function defined on the real vector space $E$ and taking values in
an order complete vector space $F.$ If $H$ is a linear subspace of $E$ and
$T_{H}:H\rightarrow F$ is a linear operator satisfying $T_{H}(x)\leq\Phi(x)$
for all $x\in H$, then there exists a linear operator $T:E\rightarrow F$ which
extends $T_{H}$ and $T(x)\leq\Phi(x)$ for all $x\in E.$
\end{theorem}

\begin{proof}
Our argument is close (but simpler) to that of Zowe \cite{Zowe}, Theorem 2.1,
p. 284.

Consider the set $\mathcal{A}$ of all pairs $(K,T_{K})$ consisting of a vector
subspace $K\subset E$ and a linear operator $T_{K}:K\rightarrow F$ such that
$K\supset H,$ $T_{K}|_{H}=T_{H}$ and $T_{K}\leq\Phi.$ The set $\mathcal{A}$ is
nonempty because it contains at least the pair $(H,T_{H})$. Notice also that
$\mathcal{A}$ is inductively ordered with respect to the order relation
$\prec$ defined by%
\[
(K,T_{K})\prec(L,T_{L})\text{ }\Longleftrightarrow\text{ }K\subset L\text{ and
}T_{L}|_{K}=T_{K}.
\]
According to Zorn's Lemma, $\mathcal{A}$ admits a maximal element, say
$(H_{M},T_{H_{M}}).$ We have to prove that $H_{M}=E.$ Indeed, if $H_{M}$ is
strictly included in $E,$ then one can choose an element $v_{0}\in E\backslash
H_{M}$ to which we associate a pair $(H_{0},T_{0})$ where $H_{0}=H_{M}%
\oplus\mathbb{R}v_{0}$ and $T_{0}:H_{0}\rightarrow F$ is the linear operator
defined via the formula%
\[
T_{0}(h+\alpha v_{0})=T_{H_{M}}(h)+\alpha y_{0}\text{ for all }h\in
H_{M},\text{ }\alpha\in\mathbb{R};
\]
here $y_{0}$ is an element of $F$ that will be specified to obtain
$(H_{M},T_{H_{M}})\prec(H_{0},T_{0}),$ a condition that is equivalent to
\begin{equation}
T_{0}(h+\alpha v_{0})=T_{H_{M}}(h)+\alpha y_{0}\leq\Phi(h+\alpha v_{0})\text{
for all }h\in H_{M},\text{ }\alpha\in\mathbb{R}. \label{zeq}%
\end{equation}
Once this accomplished, we get a pair $(H_{0},T_{0})$ that contradicts the
maximality of $(H_{M},T_{H_{M}})$. Therefore $H_{M}=E$ and that will end the
proof of the theorem.

In order to prove the existence of $y_{0}$ notice first that the inequality
(\ref{zeq}) is equivalent to%
\[
y_{0}\leq\frac{\Phi(h+\alpha v_{0})-T_{H_{M}}(h)}{\alpha}\text{\quad for all
}h\in H_{M}%
\]
if $\alpha>0,$ and it is equivalent to
\[
y_{0}\geq\frac{\Phi(h+\beta v_{0})-T_{H_{M}}(h)}{\beta}\text{\quad for all
}h\in H_{M}%
\]
if $\beta<0.$ Therefore, taking into account the order completeness of $F,$
the existence of $y_{0}$ reduces to the fact that%
\[
\frac{\Phi(h_{2}+\beta v_{0})-T_{H_{M}}(h_{2})}{\beta}\leq\frac{\Phi
(h_{1}+\alpha v_{0})-T_{H_{M}}(h_{1})}{\alpha}%
\]
whenever $h_{1},h_{2}\in H_{M},$ $\alpha>0$ and $\beta<0.$ The last inequality
is equivalent to%
\begin{equation}
T_{H_{M}}(\alpha^{-1}h_{1}-\beta^{-1}h_{2})\leq\alpha^{-1}\Phi(h_{1}+\alpha
v_{0})-\beta^{-1}\Phi(h_{2}+\beta v_{0}) \label{eqint}%
\end{equation}
and (\ref{eqint}) follows from the definition of $T_{H_{M}}$ and the convexity
of $\Phi:$
\begin{multline*}
\frac{1}{\alpha^{-1}-\beta^{-1}}\left(  \alpha^{-1}T_{H_{M}}(h_{1})-\beta
^{-1}T_{H_{M}}(h_{2})\right) \\
=T_{H_{M}}\left(  \frac{\alpha^{-1}}{\alpha^{-1}-\beta^{-1}}h_{1}+\frac
{-\beta^{-1}}{\alpha^{-1}-\beta^{-1}}h_{2}\right) \\
\leq\Phi\left(  \frac{\alpha^{-1}}{\alpha^{-1}-\beta^{-1}}h_{1}+\frac
{-\beta^{-1}}{\alpha^{-1}-\beta^{-1}}h_{2}\right) \\
=\Phi\left(  \frac{\alpha^{-1}}{\alpha^{-1}-\beta^{-1}}(h_{1}+\alpha
v_{0})+\frac{-\beta^{-1}}{\alpha^{-1}-\beta^{-1}}(h_{2}+\beta v_{0})\right) \\
\leq\frac{\alpha^{-1}}{\alpha^{-1}-\beta^{-1}}\Phi(h_{1}+\alpha v_{0}%
)+\frac{-\beta^{-1}}{\alpha^{-1}-\beta^{-1}}\Phi(h_{2}+\beta v_{0}).
\end{multline*}
This ends the proof of Theorem \ref{thm2}.
\end{proof}

\begin{corollary}
[Topological version of Theorem 2]\label{cor2} Let $\Phi$ be a convex function
defined on the Banach space $E$ and taking values in an order complete Banach
space $F$ such that $\Phi(0)=0$ and $\Phi$ is continuous at the origin. If $H$
is a linear subspace of $E$ and $T_{H}:H\rightarrow F$ is a linear operator
satisfying the inequality $T_{H}(x)\leq\Phi(x)$ for all $x\in H$, then there
exists a continuous linear operator $T:E\rightarrow F$ which extends $T_{H}$
and verifies the inequality
\begin{equation}
T(x)\leq\Phi(x)\text{\quad for all }x\in E. \label{eqsdiff0}%
\end{equation}

\end{corollary}

\begin{proof}
Indeed, the existence of a linear operator $T:E\rightarrow F$ which extends
$T_{H}$ and verifies the inequality (\ref{eqsdiff0}) follows from Theorem
\ref{thm2}\emph{. }Since\emph{ }$T$ is linear, its continuity is equivalent to
the continuity at the origin. Taking into account that%
\[
-\Phi(-x)\leq T(x)\leq\Phi(x)\text{\quad for all }x\in E,
\]
for every sequence $x_{n}\rightarrow0$ we have $-\Phi(-x_{n})\leq T(x_{n}%
)\leq\Phi(x_{n})$ and the fact that $T(x_{n})\rightarrow0$ is a consequence of
the squeeze theorem.
\end{proof}

For $H=\{0\}$ and $\Phi$ replaced by $\Phi(x+x_{0})-\Phi(x_{0})$ where $x_{0}$
is arbitrarily fixed in $E,$ Corollary 2 yields the following particular case
of Theorem 1:

\begin{corollary}
\label{cor3}Let $\Phi$ be a convex function defined on the Banach space $E$
and taking values in an order complete Banach space $F.$ Then $\partial
_{x_{0}}\Phi\neq\emptyset$ at every point $x_{0}$ of continuity of $\Phi.$
\end{corollary}

\begin{remark}
\label{rem1}As was noticed in the proof of Corollary \emph{\ref{cor2}, }an
inequality of the form $T\leq\Phi$ implies the continuity of $T$ on $E$ when
$T:E\rightarrow F$ is a linear operator and $\Phi:E\rightarrow F$ is a convex
function such that $\Phi(0)=0$ and $\Phi$ is continuous at the origin.
However, the continuity of $T$ also holds by assuming that $E$ is an ordered
Banach space and $\Phi$ is an isotone function such that $T\leq\Phi$
\emph{(}the properties of\emph{ }convexity and continuity of $\Phi$ being
unnecessary in this case\emph{)}. Indeed,%
\[
nT(-x)=T(-nx)\leq\Phi(-nx)\leq\Phi(0)\text{\quad for all }n\in\mathbb{N}\text{
and }x\in E_{+}%
\]
since $\Phi$ is isotone. Taking into account that any ordered Banach space is
Archime-dean \emph{(}see \emph{\cite{SW1999}}, Remark \emph{19} $(e)$, p.
\emph{254)}, we infer that $T(-x)\leq0$ for all $x\in E_{+},$ that is, $T$ is
a positive operator. The continuity of $T$ is now a consequence of Lemma\emph{
\ref{lem1*}}.
\end{remark}

The argument of Theorem 2 can be suitably adapted (see also \cite{Olt1978} and
\cite{Olt1983}) to prove the following variant of it:

\begin{theorem}
\label{thm3}Let $E$ be an ordered vector space, $F$ an order complete vector
space, $H\subset E$ a vector subspace and $T_{1}:H\rightarrow F$ a linear
operator. If there exists a convex function $\Phi:E_{+}\rightarrow F$ such
that $T_{1}(h)\leq\Phi(x)$ for all $(h,x)\in H\times E_{+}$ such that $h\leq
x,$ then $T_{1}$ admits a positive linear extension $T:E\rightarrow F$ such
that $T|_{E_{+}}\leq\Phi.$
\end{theorem}

It is worth noticing that the hypothesis concerning the order completeness of
the codomain $F$ is not only sufficient for the validity of Theorem 2, but it
is also necessary (in the class of ordered vector spaces). See Buskes
\cite{Bus1993} for a survey on this challenging result due to Bonnice,
Silvermann and To. It is natural to ask whether Theorem 1 admits a similar companion.

\begin{problem}
Suppose that $F$ is an ordered Banach space such that $\partial_{x_{0}}%
\Phi\neq\emptyset$ for all continuous convex functions $\Phi$ $($from an
arbitrary Banach space $E$ to $F)$ and all $x_{0}\in E.$ Is $F$ necessarily an
order complete vector space$?$
\end{problem}

A particular class of convex functions is that of sublinear operators. Suppose
that $E$ is a vector space and $F$ is an ordered vector space. A map
$P:E\rightarrow F$ is called a \emph{sublinear operator} if it is subadditive
and positively homogeneous, that is,%
\[
P(x+y)\leq P(x)+P(y)\text{ and }P(\lambda x)=\lambda P(x)
\]
for all $x,y\in E$ and $\lambda\geq0$. Every sublinear operator is convex and
a convex function $\Phi:E\rightarrow F$ is sublinear if and only if it is
positively homogeneous. In the topological context, the continuity of a
sublinear operator $P:E\rightarrow F$ is equivalent to its continuity at the
origin, which in turn is equivalent to existence of a constant $\lambda\geq0$
such that%
\[
\left\Vert P\left(  x\right)  \right\Vert \leq\lambda\left\Vert x\right\Vert
\text{\quad for all }x\in E.
\]
The smallest constant $\lambda=\left\Vert P\right\Vert $ with this property
will be called the norm of $P.$ The set $\mathcal{SL}\left(  E,F\right)  ,$ of
all continuous sublinear operators includes the Banach space $\mathcal{L}%
\left(  E,F\right)  ,$ of all linear and continuous operators from $E$ to $F,$
and is included (as a closed convex cone) in the Banach space
$\operatorname*{Lip}_{0}\left(  E,F\right)  ,$ of \ all Lipschitz maps from
$E$ to $F$ that vanish at the origin.

We end this section by noticing that in the case of continuous sublinear
operators $P:E\rightarrow F$ the subdifferential at the origin,
\[
\partial P=\partial_{0}P=\{T\in\mathcal{L}\left(  E,F\right)  :T\leq P\}
\]
also called the \emph{support }of\emph{ }$P,$ plays a special role in
emphasizing the various properties of $P.$

Here is a simple example concerning the case of an order complete Banach
lattice $E$ and the sublinear operator
\[
P:E\rightarrow E,\text{ }P(x)=\left\vert x\right\vert .
\]
The support of $P$ consists of all continuous linear operators $T\in
\mathcal{L}\left(  E,E\right)  $ such that $\left\vert T(x)\right\vert
\leq\left\vert x\right\vert $ for $x\in E.$ One can easily prove that
$\partial P$ also equals the set of all regular operators $T:E\rightarrow E$
whose absolute value $\left\vert T\right\vert $ is dominated by the identity
$I$ of $E.$ Indeed,%
\[
\left\vert T\right\vert (x)=\sup_{\left\vert y\right\vert \leq x}\left\vert
T(y)\right\vert \text{\quad for all }x\in E_{+}.
\]
See \cite{MN}, Theorem 1.3.2, p. 24. For every vector $x_{0}\neq\pm\left\vert
x_{0}\right\vert $ in $E,$ and every operator $T\in\partial_{x_{0}}P$ we have
$T(x_{0})=\left\vert x_{0}\right\vert $ and $T(x)\leq\left\vert x\right\vert $
for all $x.$ Then $T(x_{0}^{+})-T(x_{0}^{-})=x_{0}^{+}+x_{0}^{-},$ whence%
\[
0\leq x_{0}^{+}-T(x_{0}^{+})=T(-x_{0}^{-})-x_{0}^{-}=T(-x_{0}^{-})-\left\vert
-x_{0}^{-}\right\vert \leq0,
\]
that is, $x_{0}^{+}=T(x_{0}^{+})$ and $-x_{0}^{-}=T(x_{0}^{-}).$ As a
consequence, $-1$ and $1$ are eigenvalues to which correspond positive
eigenvectors. Since $T\left(  \left\vert x_{0}\right\vert \right)  =x_{0}$ and
$T(x_{0})=\left\vert x_{0}\right\vert ,$ each operator $T\in\partial_{x_{0}}P$
admits also a periodic orbit $\{x_{0},\left\vert x_{0}\right\vert \}$ of
length \emph{2}.

The dynamics of linear operators belonging to the support of a continuous
sublinear operator can be rather intricate due to the existence of points with
dense orbits. So is the case of \emph{Ces\`{a}ro operator},%
\[
\operatorname*{Ces}:L^{p}(0,1)\rightarrow L^{p}(0,1),\text{\quad}\left(
\operatorname*{Ces}f\right)  (x)=\frac{1}{x}\int_{0}^{x}f(t)\mathrm{d}t,
\]
whose boundedness for $p\in(1,\infty)$ follows from Hardy's inequality. See
Bayart and Matheron \cite{BM}, Exercise 1.2, p. 28. This positive linear
operator belongs to the support of the continuous sublinear operator%
\[
\left(  Pf\right)  (x)=\left\vert \frac{1}{x}\int_{0}^{x}f(t)\mathrm{d}%
t\right\vert .
\]

\section{A characterization of the isotone convex functions}

A classical result due to Stolz asserts that every real-valued convex function
$f$ defined on an open real interval $I$ admits finite sided derivatives and
in addition%
\[
f_{-}^{\prime}(x)\leq f_{+}^{\prime}(x)\leq\frac{f(y)-f(x)}{y-x}\leq
f_{-}^{\prime}(y)\leq f_{+}^{\prime}(y)
\]
for all $x<y$ in $I;$ see \cite{NP2018}, Theorem 1.4.2, p. 25. The subgradient
inequality yields the following generalization:

\begin{lemma}
\label{lem3}Suppose that $C$ is an open convex subset of a Banach space $E,$
$F$ is an order complete Banach space and $\Phi:C\rightarrow F$ is a
continuous convex function. Then for every pair of points $x$ and $y$ in $C$
and every operators $T\in\partial_{x}\Phi$ and $S\in\partial_{y}\Phi$ we have%
\[
T(y-x)\leq\Phi(y)-\Phi(x)\leq S(y-x).
\]

\end{lemma}

This enables us to characterize the isotonicity of continuous convex functions
defined on the whole space in a similar way to the case of differentiable
functions defined on real intervals.

\begin{theorem}
\label{thm4}Suppose that $E$ is an ordered Banach space$,$ $F$ is an order
complete Banach space and $\Phi:E\rightarrow F$ is a continuous convex
function. Then the following statements are equivalent:

$(a)$ $x\leq y$ in $E$ implies $\Phi(x)\leq\Phi(y);$

$(b)$ the subdifferential of $\Phi$ at any point $x_{0}\in E$ consists of
positive linear operators$;$

$(c)$ for each $x_{0}\in E,$ there exists a positive linear operator
$T\in\partial_{x_{0}}\Phi.$
\end{theorem}

\begin{proof}
$(a)\Rightarrow(b)$ If $T\in\partial_{x_{0}}\Phi,$ then
\[
T(x)\leq\Phi(x)-\Phi(x_{0})+T(x_{0})\text{\quad for all }x\in E.
\]
Since $\Phi$ is supposed to be isotone, we have%
\[
nT(-x)=T(-nx)\leq\Phi(-nx)-\Phi(x_{0})+T(x_{0})\leq\Phi(0)-\Phi(x_{0}%
)+T(x_{0})
\]
for all $n\in\mathbb{N}$ and $x\in E_{+}.$ Since any ordered Banach space is
Archimedean, this implies $T(-x)\leq0$ and thus $T$ is a positive operator.

The implication $(b)\Rightarrow(c)$ is trivial, while the implication
$(c)\Rightarrow(a)$ follows from the left-hand side inequality in Lemma
\ref{lem3}.
\end{proof}

The isotonicity on the positive cone admits a similar characterization. We
shall need the following result about controlled regularity.

\begin{lemma}
\label{lem4}Suppose that $E$ is an ordered vector space, $F$ is an order
complete vector lattice and $\Psi:E_{+}\rightarrow F$ is a convex function.
Then for every linear operator $S:E\rightarrow F$ the following two assertions
are equivalent:

$(a)$ there exist two positive linear operators $T,U:E\rightarrow F$ such that
$S=T-U$ and $T|_{E_{+}}\leq\Psi;$

$(b)$ $S(x_{1})\leq\Psi(x_{2})$ for all $x_{1},x_{2}\in E$ such that $0\leq
x_{1}\leq x_{2}.$
\end{lemma}

For details, see \cite{Olt1983}.

\begin{theorem}
\label{thm5}Suppose that $E$ is an ordered Banach space, $F$ is an order
complete Banach lattice and $\Phi:E_{+}\rightarrow F$ is a continuous convex
function. If the interior $\operatorname*{int}E_{+}$ of the positive cone of
$E$ is nonempty, $\Phi$ is isotone on $\operatorname*{int}E_{+}$ if and only
if the subdifferential of $\Phi$ at each point $x_{0}\in\operatorname*{int}%
E_{+}$ contains a positive linear operator.
\end{theorem}

\begin{proof}
For the necessity part, suppose that $\Phi$ is isotone and let $x_{0}%
\in\operatorname*{int}E_{+}$ arbitrarily fixed. According to Theorem
\ref{thm1} there exists $S\in\mathcal{L}(E,F)$ such that $\Phi(x)\geq
\Phi(x_{0})+S(x-x_{0})$ for every $x\in E_{+}.$ Therefore if $x_{1},x_{2}\in
E_{+}$ and $0\leq x_{1}\leq x_{2}$ we have
\[
S(x_{1})\leq\Phi(x_{1})-\Phi(x_{0})+S(x_{0})\leq\Phi(x_{2})-\Phi
(x_{0})+S(x_{0}).
\]
The function $\Psi:E_{+}\rightarrow F$ defined by the formula $\Psi
(x)=\Phi(x)-\Phi(x_{0})+S(x_{0})$ is isotone, continuous and convex and the
pair $(S,\Psi)$ fulfils the assumptions of Lemma \ref{lem4}. As a consequence,
$S$ can be decomposed as the difference $T-U$ of two positive linear operators
with $T|_{E_{+}}\leq\Psi.$ We have
\begin{equation}
T(x)\leq\Psi(x)=\Phi(x)-\Phi(x_{0})+S(x_{0}) \label{eqx}%
\end{equation}
for all $x\in E_{+}.$ On the other hand,%
\[
T(x_{0})=S(x_{0})+U(x_{0})\geq S(x_{0})\text{ and }T(x_{0})\leq\Psi
(x_{0})=S(x_{0}),
\]
whence $S(x_{0})=T(x_{0}).~$According to (\ref{eqx}), we conclude that
\[
\Phi(x)\geq\Phi(x_{0})+T(x-x_{0})\text{\quad for all }x\in E_{+},
\]
that is, $T\in\partial_{x_{0}}\Phi$.

The sufficiency part of Theorem \ref{thm5} is obvious.
\end{proof}

At the moment we don't know whether Theorem \ref{thm5} holds in the case where
$F$ is only an order complete Banach space.

\begin{remark}
Under the assumptions of Theorem \emph{\ref{thm5}}, the isotonicity of $\Phi$
on $\operatorname*{int}E_{+}$ implies its isotonicity on $E_{+}$. Indeed,
suppose that $x,y\in E_{+}$ and $x\leq y$. For $x_{0}\in\operatorname*{int}%
E_{+}$ arbitrarily fixed and $t\in\lbrack0,1),$ both elements $u_{t}%
=x_{0}+t(x-x_{0})$ and $v_{t}=x_{0}+t(y-x_{0})$ belong to $\operatorname*{int}%
E_{+}$ and $u_{t}\leq v_{t}.$ Moreover, $u_{t}\rightarrow x$ and
$v_{t}\rightarrow y$ as $t\rightarrow1.$ Passing to the limit in the
inequality $\Phi(u_{t})\leq\Phi(v_{t})$ we conclude that $\Phi(x)\leq\Phi(y).$
\end{remark}

An important case when $\operatorname*{int}E_{+}\neq\emptyset$ is that of
ordered Banach spaces $E$ whose norm is associated to a strong order unit, for
example $c,$ $\ell^{\infty},$ $C([0,1]),$ $L^{\infty}(\mathbb{R}^{n})$ and
$\operatorname*{Sym}(n,\mathbb{R})$. See Amann \cite{Amann1976} for more examples.

The positive cone of each of the Lebesgue spaces $L^{p}(\mathbb{R}^{n})$ with
$1\leq p<\infty$ has empty interior. For such spaces the isotonicity on the
positive cone admits the following characterization.

\begin{theorem}
\label{thm6}Suppose that $E$ is an ordered Banach space and $F$ is an order
complete Banach space. Then a continuous convex function $\Phi:E_{+}%
\rightarrow F$ is isotone if \emph{(}and only if\emph{) }for every $x_{0}\in
E_{+},$ there exists a positive linear operator $T\in\mathcal{L}(E,F)$ such
that
\[
\Phi(x)-\Phi(x_{0})\geq T(x-x_{0})\text{\quad for all }x\geq x_{0}.
\]

\end{theorem}

\begin{proof}
The necessity part follows from Theorem \ref{thm3} applied to $H=\{0\},$
$T_{1}=0$ and $\Phi$ replaced by $\tilde{\Phi}(x)=\Phi(x+x_{0})-\Phi(x_{0}).$
Since $\tilde{\Phi}$ is isotone and $\tilde{\Phi}(0)=0,$ the assumptions of
Theorem \ref{thm3} are fulfilled, so there exists a positive and linear
operator $T:E\rightarrow F$ such that $T|_{E_{+}}\leq\tilde{\Phi},$ that is,%
\[
\Phi(x+x_{0})-\Phi(x_{0})\geq T(x)\text{\quad for all }x\geq0.
\]

The sufficiency part is obvious.
\end{proof}

\section{Examples of Isotone Convex Functions}

\begin{example}
\label{ex1}Suppose that $E$ is a Banach lattice, $F$ is an order complete
Banach lattice and $A:E\rightarrow F$ is a positive linear operator. Then
\[
\Phi\left(  x\right)  =A\left(  x^{+}\right)  ,\text{\quad}x\in E
\]
defines a continuous sublinear operator which is monotone on $E$. The support
of $\Phi$ consists of all positive linear operators $T\in\mathcal{L}(E,F)$
dominated by $A.$ Theorem \emph{\ref{thm4}} asserts the striking fact that for
every $x_{0}\in E,$ $x_{0}\neq0,$ there exists a positive linear operator
$T\in\mathcal{L}(E,F)$ such that%
\[
T(x_{0})=A\left(  x_{0}^{+}\right)  \text{\quad and\quad}T(x)\leq A\left(
x^{+}\right)  \text{ for all }x\in E.
\]

\end{example}

\begin{example}
\label{ex2}The seminorm $\Phi:\ell^{\infty}\rightarrow\mathbb{R}$ defined by
\[
\Phi((x_{n})_{n})=\underset{n\rightarrow\infty}{\lim\sup}\left\vert
x_{n}\right\vert ,
\]
is continuous and nowhere G\^{a}teaux differentiable \emph{(}see
\cite{Ph1993}, Example \emph{1.21}, p. \emph{13) }and also isotone on the
positive\emph{ }cone\emph{ (}the set of all bounded sequences with nonnegative
coefficients\emph{). }According to Theorem \ref{thm5}, for every positive
sequence $(a_{n})_{n}\in\ell_{+}^{\infty}$ there exists a positive linear
functional $T:\ell^{\infty}\rightarrow\mathbb{R}$ such that
\[
T\left(  (a_{n})_{n}\right)  =\underset{n\rightarrow\infty}{\lim\sup}%
a_{n}\text{\quad and\quad}T\left(  (x_{n})_{n}\right)  \leq\underset
{n\rightarrow\infty}{\lim\sup}x_{n}\text{\quad for all }(x_{n})_{n}\in\ell
_{+}^{\infty}.
\]
Notice that $T$ cannot be an element of $\ell^{1},$ the predual of
$\ell^{\infty},$ and its existence is assured by the Axiom of choice.
\end{example}

\begin{example}
\label{ex3}Suppose that $E=C(K)$ \emph{(}where $K$ is a compact Hausdorff
space\emph{)} and $F$ is the Lebesgue space $L^{p}(m)$ associated to a
$\sigma$-additive probability measure $m$ and to an index $p\in\lbrack
1,\infty].$ As was noticed by Rubinov\emph{ \cite{Rub1977}}, p.\emph{ 130,}
the support of the sublinear and monotone operator
\[
\Phi:C\left(  K\right)  \rightarrow L^{p}(m),\text{\quad}\Phi(f)=\left(
\sup_{t\in K}f(t)\right)  \cdot1.
\]
consists of the Markov operators, that is,
\[
\partial\Phi=\left\{  T\in\mathcal{L}\left(  C\left(  K\right)  ,L^{p}%
(m)\right)  :T\geq0,\text{ }T(1)=1\right\}  .
\]
Indeed, changing $f$ to $-f$, we infer that every $T\in\partial\Phi$ also
verifies%
\[
\left(  \inf_{t\in K}f(t)\right)  \cdot1\leq Tf.
\]

According to Theorem \emph{\ref{thm4}}, for every function $h\in C\left(
K\right)  ,$ $h\neq0,$ there exists a Markov operator $T$ such that%
\[
T(h)=\left(  \sup_{t\in K}h(t)\right)  \cdot1\text{ and }T(f)\leq\left(
\sup_{t\in K}f(t)\right)  \cdot1\text{ for all }f\in C\left(  K\right)  .
\]

\end{example}

In the next example $E$ is the regularly ordered Banach space
$\operatorname*{Sym}(n,\mathbb{R}),$ of all $n\times n$-dimensional symmetric
matrices with real coefficients. The norm of a symmetric matrix $A$ is defined
by the formula%
\[
\left\Vert A\right\Vert =\sup_{\left\Vert x\right\Vert \leq1}\left\vert
\langle Ax,x\rangle\right\vert
\]
and the positive cone $\operatorname*{Sym}^{+}(n,\mathbb{R})$ of
$\operatorname*{Sym}(n,\mathbb{R})$ consists of all symmetric matrices $A$
such that $\langle Ax,x\rangle\geq0$ for all $x.$ As is well known, $A\geq0$
if and only if its spectrum $\sigma(A)$ is included in $[0,\infty).$ For
convenience, we will list the eigenvalues of a symmetric matrix $A\in
\operatorname*{Sym}(n,\mathbb{R})$ in decreasing order and repeated according
to their multiplicities as follows:%
\[
\lambda_{1}^{\downarrow}(A)\geq\cdots\geq\lambda_{n}^{\downarrow}(A).
\]
According to Weyl's monotonicity\emph{ }principle\emph{ }$($see
\emph{\cite{NP2018}}, Corollary \emph{4.4.3}, p. \emph{203}$)$,%
\[
A\leq B\text{ in }\operatorname*{Sym}(n,\mathbb{R})\text{ if and only if
}\lambda_{i}^{\downarrow}(A)\leq\lambda_{i}^{\downarrow}(B)\text{ for
}i=1,...,n.
\]

\begin{example}
\label{ex4}The map
\[
\Phi:\operatorname*{Sym}(n,\mathbb{R})\rightarrow\mathbb{R}^{n},\text{\quad
}\Phi(A)=\lambda_{1}^{\downarrow}(A)\cdot\mathbf{1},\text{ }%
\]
where $\mathbf{1}$ denotes the vector in $\mathbb{R}^{n}$ with all components
equal to 1, is sublinear and continuous. The sublinearity is a consequence of
the fact that
\[
\lambda_{1}^{\downarrow}(A)=\sup_{\left\Vert x\right\Vert =1}\langle
Ax,x\rangle,
\]
while the continuity follows from Weyl's\emph{ }perturbation theorem, which
asserts that for every pair of matrices $A,B$ in $\operatorname*{Sym}%
(n,\mathbb{R})$ we have
\[
\max_{1\leq k\leq n}|\lambda_{k}^{\downarrow}(A)-\lambda_{k}^{\downarrow
}(B)|\leq\Vert A-B\Vert.
\]
See\emph{ \cite{NP2018}}, Theorem \emph{4.4.7}, p. \emph{205}. From Weyl's
monotonicity principle,\emph{ }we easily infer that the sublinear operator
$\Phi$ is isotone on the whole space $\operatorname*{Sym}(n,\mathbb{R}).$

Every bounded linear operator $T:\operatorname*{Sym}(n,\mathbb{R}%
)\rightarrow\mathbb{R}^{n}$ in the support of $\Phi$ is positive $($since
$A\leq0$ implies $T(A)\leq\Phi(A)\leq0)$ and maps the identity matrix
$\mathrm{I}$ into the vector $\mathbf{1}$ \emph{(}since $T(\mathrm{I})\leq
\Phi(\mathrm{I})=\mathbf{1}$ and $T(\mathrm{I})\geq-\Phi(-\mathrm{I}%
)=\mathbf{1}).$ Conversely, if $T:\operatorname*{Sym}(n,\mathbb{R}%
)\rightarrow\mathbb{R}^{n}$ is a positive linear operator such that
$T(\mathrm{I})=\mathbf{1},$ then%
\[
A\leq\lambda_{1}^{\downarrow}(A)\mathrm{I}\Longrightarrow T(A)\leq\lambda
_{1}^{\downarrow}(A)\cdot T(\mathrm{I})=\lambda_{1}^{\downarrow}%
(A)\cdot\mathbf{1}=\Phi(A),
\]
that is, $T\in\partial\Phi.$

An important example of operator belonging to the subdifferential of $\Phi$ at
$\mathrm{I}$ is that which associates to each symmetric matrix $A$ the vector
$\operatorname*{diag}(A),$ displaying the main diagonal of $A$.

According to Theorem \emph{\ref{thm4}}, for every $A\in\operatorname*{Sym}%
(n,\mathbb{R}),$ $A\neq0,$ there exists a positive linear operator
$T:\operatorname*{Sym}(n,\mathbb{R})\rightarrow\mathbb{R}^{n}$\ such that
$T(A)=\lambda_{1}^{\downarrow}(A)\cdot\mathbf{1}$ and $T(B)\leq\lambda
_{1}^{\downarrow}(B)\cdot\mathbf{1}$ for all $B\in\operatorname*{Sym}%
(n,\mathbb{R}).$
\end{example}

Related to the above discussion is the subject of matrix concave/matrix
monotone functions that offers nice examples of isotone concave functions. The
basic ingredient is the functional spectral calculus with continuous
functions. See \cite{NP2018}, Section 5.2, for details. A continuous function
$f:[0,\infty)\rightarrow\lbrack0,\infty)$ is called \emph{matrix concave} if%
\begin{equation}
f((1-\lambda)A+\lambda B)\geq(1-\lambda)f(A)+\lambda f(B) \label{opconv}%
\end{equation}
for every $\lambda\in\lbrack0,1]$ and every pair of matrices $A$,
$B\in\operatorname*{Sym}(n,\mathbb{R})$ with spectra in $[0,\infty)$; the
function $f$ is called \emph{matrix monotone} if, for every pair of matrices
$A$, $B\in\operatorname*{Sym}(n,\mathbb{R})$ with spectra in $[0,\infty),$ the
following implication holds:%
\[
A\leq B\text{ implies }f(A)\leq f(B).
\]
Under these circumstances, one can prove that a continuous function from
$[0,\infty)$ into itself is matrix concave if and only if it is matrix
monotone. See \cite{NP2018}, Theorems 5.2.6 and 5.2.7, p.236. An example
illustrating this fact is provided by the function%
\[
\Phi:\operatorname*{Sym}\nolimits^{+}(n,\mathbb{R})\mathbb{\rightarrow
}\operatorname*{Sym}(n,\mathbb{R}),\text{\quad}\Phi(A)=A^{p},
\]
where $p\in(0,1]$ is a parameter. The matrix monotonicity of the function
$t^{p}$ (for $t>0$ and $0<p<1$) is known as the L\"{o}wner-Heinz inequality. A
short proof of it can be found in \cite{Ped}.

\section{The Case of Choquet Integral}

An interesting example of continuous sublinear operator which is monotone on
the whole space is provided by Choquet integral, which was introduced by
Choquet in \cite{Ch1954}.For the convenience of the reader we will briefly
recall some very basic facts concerning Choquet's theory of integrability with
respect to an isotone set function (not necessarily additive). Full details
are to be found in the books of Denneberg \cite{Denn}, Grabish \cite{Gr2016}
and Wang and Klir \cite{WK}.

Given be a compact Hausdorff space $X$ we attach to it the $\sigma$-algebra
$\mathcal{B}(X)$ of all Borel subsets of $X$ and the Banach lattice $C(X)$ of
all real-valued continuous functions defined on $X.$

\begin{definition}
\label{def1}A set function $\mu:\mathcal{B}(X)\rightarrow\mathbb{R}_{+}$ is
called a capacity if it verifies the following two conditions:

$(a)$ $\mu(\emptyset)=0$ and $\mu(X)=1;$

$(b)~\mu(A)\leq\mu(B)$ for all $A,B\in\mathcal{B}(X)$, with $A\subset B$.

A capacity $\mu$ is called submodular \emph{(}or strongly subadditive\emph{)}
if
\begin{equation}
\mu(A\bigcup B)+\mu(A\bigcap B)\leq\mu(A)+\mu(B) \label{eqsubmod}%
\end{equation}
for all $A,B\in\mathcal{B}(X).$
\end{definition}

A simple way to construct nontrivial examples of submodular capacities is to
start with a probability measure $P:\mathcal{B}(X)\rightarrow\lbrack0,1]$ and
to consider any nondecreasing concave function $u:[0,1]\rightarrow\lbrack0,1]$
such that $u(0)=0$ and $u(1)=1;$ for example, one may chose $u(t)=t^{a}$ with
$0<\alpha<1.$ We have%
\[
P(A\bigcup B)+P(A\bigcap B)=P(A)+P(B)
\]
and%
\[
P(A\bigcap B)\leq P(A),P(B)\leq P(A\bigcup B),
\]
for all $A,B\in\mathcal{B}(X),$ so by the Hardy-Littlewood-P\'{o}lya
Inequality of majorization (see \cite{NP2018}, Theorem 4.1.3, p. 186) we infer
that the \emph{distorted probability} $\mu=u\circ P$ verifies the equation
(\ref{eqsubmod}) that is, $\mu$ is a submodular capacity.

\begin{definition}
\label{def2}The Choquet integral of a Borel measurable function
$f:X\rightarrow\mathbb{R}$ on a set $A\in\mathcal{B}(X)$ is defined by the
formula
\begin{align}
(\operatorname*{C})\int_{A}f\mathrm{d}\mu &  =\int_{0}^{+\infty}\mu\left(
\{x\in X:f(x)\geq t\}\cap A\right)  \mathrm{d}t\nonumber\\
&  +\int_{-\infty}^{0}\left[  \mu\left(  \{x\in X:f(x)\geq t\}\cap A\right)
-\mu(A)\right]  \mathrm{d}t, \label{Ch}%
\end{align}
where the integrals in the right hand side are generalized Riemann integrals.

If $(C)\int_{A}f\mathrm{d}\mu$ exists in $\mathbb{R}$, then $f$ is called
Choquet integrable on $A$.
\end{definition}

Notice that if $f\geq0$, then the last integral in the formula (\ref{Ch}) is 0.

The inequality sign $\geq$ in the above two integrands can be replaced by $>;$
see \cite{WK}, Theorem 11.1,\emph{ }p. 226.

The Choquet integral agrees with the Lebesgue integral in the case of
probability measures. See Denneberg \cite{Denn}, p. 62.

The basic properties of the Choquet integral make the objective of the
following two remarks.

\begin{remark}
\label{rem2}\emph{(Denneberg \cite{Denn}, Proposition 5.1, p. 64) }$(a)$ If
$f$ and $g$ are Choquet integrable on the Borel set $A$, then
\begin{gather*}
f\geq0\text{ implies }(\operatorname*{C})\int_{A}f\mathrm{d}\mu\geq0\text{
\quad\emph{(}positivity\emph{)}}\\
f\leq g\text{ implies }\left(  \operatorname*{C}\right)  \int_{A}%
f\mathrm{d}\mu\leq\left(  \operatorname*{C}\right)  \int_{A}g\mathrm{d}%
\mu\text{ \quad\emph{(}monotonicity\emph{)}}\\
\left(  \operatorname*{C}\right)  \int_{A}af\mathrm{d}\mu=a\cdot\left(
\operatorname*{C}\right)  \int_{A}f\mathrm{d}\mu\text{ for all }a\geq0\text{
\quad\emph{(}positive\emph{ }homogeneity\emph{)}}\\
\left(  \operatorname*{C}\right)  \int_{A}1~\mathrm{d}\mu=\mu(A)\text{\quad
\emph{(}calibration\emph{).}}%
\end{gather*}
$(b)$ In general, the Choquet integral is not additive but \emph{(}as was
noticed by Dellacherie \emph{\cite{Del1970})}, if $f$ and $g$ are comonotonic
\emph{(}that is, $(f(\omega)-f(\omega^{\prime}))\cdot(g(\omega)-g(\omega
^{\prime}))\geq0$, for all $\omega,\omega^{\prime}\in A$\emph{), }then
\[
\left(  \operatorname*{C}\right)  \int_{A}(f+g)\mathrm{d}\mu=\left(
\operatorname*{C}\right)  \int_{A}f\mathrm{d}\mu+\left(  \operatorname*{C}%
\right)  \int_{A}g\mathrm{d}\mu.
\]
An immediate consequence is the property of translation invariance,
\[
\left(  \operatorname*{C}\right)  \int_{A}(f+c)\mathrm{d}\mu=\left(
\operatorname*{C}\right)  \int_{A}f\mathrm{d}\mu+c\cdot\mu(A)
\]
for all $c\in\mathbb{R}$ and $f$ integrable on $A.$
\end{remark}

\begin{remark}
\label{rem3}\emph{(The Subadditivity Theorem) }If $\mu$ is a submodular
capacity, then\ the associated Choquet integral is subadditive, that is,%
\[
\left(  \operatorname*{C}\right)  \int_{A}(f+g)\mathrm{d}\mu\leq\left(
\operatorname*{C}\right)  \int_{A}f\mathrm{d}\mu+\left(  \operatorname*{C}%
\right)  \int_{A}g\mathrm{d}\mu
\]
for all $f$ and $g$ integrable on $A.$ See\emph{ \cite{Denn}, }Theorem\emph{
6.3, }p\emph{. 75. }In addition, the following two integral analogs of the
modulus inequality hold true,
\[
|(\operatorname*{C})\int_{A}f\mathrm{d}\mu|\leq(\operatorname*{C})\int
_{A}|f|\mathrm{d}\mu
\]
and
\[
|(\operatorname*{C})\int_{A}f\mathrm{d}\mu-(\operatorname*{C})\int
_{A}g\mathrm{d}\mu|\leq(\operatorname*{C})\int_{A}|f-g|\mathrm{d}\mu;
\]
the last assertion being covered by Corollary \emph{6.6}, p. \emph{82}, in
\emph{\cite{Denn}.}
\end{remark}

According to Remarks \ref{rem2} and \ref{rem3}, the Choquet integral defines a
continuous sublinear functional%
\[
p:C(X)\rightarrow\mathbb{R},\text{\quad}p(f)=(\operatorname*{C})\int
_{A}f\mathrm{d}\mu,
\]
which is isotone on the whole space. This example admits an immediate vector
valued companion $P:C(X)\rightarrow\mathbb{R}^{n},$ associated by a family
$\mu_{1},...,\mu_{n}$ of submodular capacities on $\mathcal{B}(X)$ via the
formula
\[
\text{\quad}P(f)=\left(  (C)\int_{X}f\mathrm{d}\mu_{1},...,(C)\int
_{X}f\mathrm{d}\mu_{n}\right)  .
\]
According to Theorem \ref{thm4}, the following result relating the Choquet
integral to Lebesgue integral holds true.

\begin{theorem}
\label{thm7}For each $h\in C(X),$ $h\neq0,$ there exists a linear and positive
operator $T:C(X)\rightarrow\mathbb{R}^{n}$ such that $T\leq P$ and
$T(h)=P(h).$
\end{theorem}

Taking into account the\ Riesz-Kakutani representation theorem for positive
linear functionals on $C(K)$ (see \cite{Bog}, Theorem 7.10.4, p. 111), we
infer from Remark \ref{rem2}, Remark \ref{rem3}, and Theorem \ref{thm7} that
actually every submodular capacity $\boldsymbol{\mu}:$ $\mathcal{B}%
(K)\rightarrow\mathbb{R}^{n}$ is the least upper bound of all positive,
regular and $\sigma$-additive vector measures $m$ dominated by
$\boldsymbol{\mu}.$ This improves on Proposition 10.3 in \cite{Denn}. Related
results with applications to inequalities are presented by Mesiar, Li and Pap
\cite{MLP2010}.

Approximation theory emphasizes many other examples of continuous, sublinear
and isotone operators. See \cite{Gal-Mjm}, \cite{GN2020a} and \cite{GN2020b}.
We mention here the case of the \emph{Bernstein-Kantorovich-Choquet }operator
$K_{n,\mu}:C([0,1])\rightarrow C([0,1]),$ associated to a submodular capacity
$\mu$ via the formula%
\[
K_{n,\mu}(f)(x)=\sum_{k=0}^{n}\frac{(C)\int_{k/(n+1)}^{(k+1)/(n+1)}f(t)d\mu
}{\mu([k/(n+1),(k+1)/(n+1)])}\cdot{\binom{n}{k}}x^{k}(1-x)^{n-k}.
\]

\medskip

\noindent\textbf{\noindent Acknowledgement. }The authors would like to thank
Professor \c{S}tefan Cobza\c{s} from the Babe\c{s}-Bolyai University of
Cluj-Napoca, for several improvements on the original version of this paper.

\end{document}